\documentclass[a4paper]{article}
\usepackage{amsthm,amssymb,amsmath,enumerate,graphicx,psfrag}
\usepackage{algorithm}

\newtheorem{definition}{Definition}

\newtheorem{theorem}[definition]{Theorem}

\newtheorem{lemma}[definition]{Lemma}
\newtheorem{conjecture}[definition]{Conjecture}

\def\henning#1{}%
\def\maya#1{}%

\newcommand{\comment}[1]{}

\newcommand{\emtext}[1]{\text{\em #1}}

\newcommand{\sm}{\setminus}

\title{Minimal bricks have many vertices of small degree}
\author{Henning Bruhn and Maya Stein\footnote{Supported by Fondecyt grant 11090141.}}
\date{}

\begin{document}
\maketitle

\begin{abstract}
We prove that  every minimal brick on $n$ vertices has at least $ n/9$ vertices of degree at most $4$.
\end{abstract}

\section{Introduction}

A key element in matching theory is the notion of a brick. We briefly
and somewhat informally explain this notion and its role. For 
a much more detailed treatment we refer to the books of Lov\'asz' and Plummer~\cite{LP86} and Schrijver~\cite{LexBible}.

A \emph{matching} (a set of independent edges)
of a graph is \emph{perfect}
if every vertex is incident with a matching edge.
Consider a matching covered graph, that is a graph in which every
edge lies in some perfect matching. A tight cut of such a graph is 
a cut that meets every perfect matching in precisely one edge.  
Contracting one, or the other, side of a tight cut $F$ we obtain two new graphs (which preserve the perfect matching structure we had in the original graph). This operation is called a `split along the tight cut $F$'. 

Clearly, we can go on splitting along tight cuts in the newly obtained graphs until arriving at graphs that contain no tight cuts. 
It was shown by Lov\'asz~\cite{BricBracDecomp} that no matter in which order we choose the tight cuts we split along, we will essentially always arrive at the same set of cuts  and graphs. The obtained decomposition is generally called a `brick and brace decomposition' because the set of final graphs (without tight cuts) is divided into those that are bipartite -- called \emph{braces} -- and those that are not -- the \emph{bricks}. This decomposition allows to reduce several problems from matching theory  to  bricks (e.g.~a graph is Pfaffian if and only if its bricks are).

Both bricks and braces have been characterised by Edmonds, Lov\'asz and Pulleyblank~\cite{ELP} in other terms. We omit the characterisation of braces. For the one 
of bricks, let us first say that 
a graph $G$ is \emph{bicritical} if $G-\{u,v\}$ has a perfect matching
for every choice of distinct vertices $u$ and $v$. 
Now bricks are precisely the bicritical and $3$-connected graphs~\cite{ELP}. 
For practical purposes let us consider a brick to be defined this way.

\medskip

The focus of this paper lies on \emph{minimal bricks}: Those bricks $G$ 
for which $G-e$ ceases to be a brick  for every edge $e\in E(G)$.
Minimality often leads to sparsity in some respect. Minimal bricks are no exception:
It is known~\cite{NoTh06}  that any minimal brick on $n$ vertices has average degree at most $5-7/n$, unless it is one of four special bricks (the prism or the wheel $W_n$ for $n=4,6,8$). 
While thus minimal bricks do have vertices of degree $3$ or $4$, they may conceivably 
be very few in number, if the average degree is very close to $5$.
Of particular interest are vertices that attain the smallest degree
possible, which is $3$ for a brick.

De Carvalho, Lucchesi and Murty~\cite{CLM06} proved that any minimal brick 
contains a vertex of degree~$3$, which had been conjectured earlier by Lov\'asz; see~\cite{CLM06}. This was extended by Norine and Thomas, who showed the existence of $3$ such vertices, and then went on to pose the following stronger conjecture. 

\begin{conjecture}[Norine and Thomas~\cite{NoTh06}]\label{ratioconj}
There is an $\alpha>0$ so that every minimal brick $G$ contains
at least $\alpha |V(G)|$ vertices of degree~$3$.
\end{conjecture}

Our main result yields further evidence for this conjecture. 

\begin{theorem}\label{thmdeg4}
Every minimal brick $G$ 
has at least $\frac 19|V(G)|$  vertices of degree at most~$4$.
\end{theorem}

We hope that the methods developed here, if substantially strengthened, 
will be useful for attacking Norine and Thomas' conjecture.

\section{Brick generation}

For practical purposes, the abstract definition of a brick as a 3-connected and bicritical graph may sometimes be
less useful than  knowing how to obtain a brick from another brick by a small local operation. De Carvalho, Lucchesi and Murty~\cite{CLM06} study such operations, and prove that any brick other than the Petersen graph can be obtained by performing these operations successively, starting with either $K_4$ or the prism. (In particular, every graph in this sequence is a brick.) Norine and Thomas~\cite{NoTh07} show a gene\-ra\-lisation of this result, which they obtained independently.

In particular, every brick has a generating sequence of ever larger bricks.
To be useful in induction proofs about minimal bricks, however, 
it appears necessary that all intermediate graphs are minimal as well, 
which is unfortunately not guaranteed by the results above. 
To mend this situation,
Norine and Thomas~\cite{NoTh06} introduce another family of operations, 
called \emph{strict extensions}, which we shall describe below. Using strict extensions, they find that each minimal brick has a generating sequence consisting only of minimal bricks:

\begin{theorem}[Norine and Thomas~\cite{NoTh06}]\label{generatingminbricks}
Every minimal brick other than the Petersen graph can be 
obtained by strict extensions starting from $K_4$ or the
prism, where all intermediate graphs are minimal bricks. 
\end{theorem}

Notice that although a strict extension of a brick is a brick, a strict extension of a minimal brick need not be a minimal brick~\cite{NoTh06}.

\medskip

Let us now formally define strict extensions, following Norine and Thomas~\cite{NoTh06}.
There are five types of strict extensions: Strict linear, bilinear, pseudolinear,
quasiquadratic and quasiquartic extensions. The first three of these are based
on an even simpler operation, the bisplitting of a vertex.

For this, consider a graph $H$ and one of its vertices $v$ of degree at least~$4$.
Partition the neighbourhood of $v$ into two sets $N_1$ and $N_2$ such that
each contains at least two vertices. We now replace $v$ by two new independent vertices, $v_1$
and $v_2$, where $v_1$ is incident with the vertices in $N_1$ and $v_2$
with the ones in $N_2$. Finally, we add a third new vertex $v_0$ that
is adjacent to precisely $v_1$ and $v_2$. We say that any such graph $H'$
is obtained from $H$ by \emph{bisplitting $v$}. The vertex $v_0$ 
is the \emph{inner vertex} of the bisplit, while $v_1$ and $v_2$ 
are the \emph{outer vertices}. Any time we perform a bisplit at 
a vertex $v$ we will tacitly assume $v$ to have degree at least~$4$.

\medskip
We will now define turn by turn the strict extensions. At the same 
time we will specify a small set of vertices, the \emph{fundament} of the strict extension. 
One should think of the fundament as a minimal set of vertices
that needs to be present, should we want to perform the extension 
in some other, usually smaller, graph.

Let $v$ be a  vertex of a graph $G$. 
We say that $G'$ is a 
\emph{strict
%\footnote{In case the reader is wondering whether there are also \emph{non-strict} linear extensions: Norine and Thomas define these as well, namely as simply adding an edge between a pair of non-adjacent neighbours.}
 linear extension of $G$}
if $G'$ is obtained by one of the three following operations. (See Figure~\ref{strlinext} for an illustration.)
\begin{enumerate}
\item We perform a bisplit at $v$, denote by $v_0$ the inner vertex, and by $v_1$ and~$v_2$ the outer vertices
of the bisplit. Choose a vertex $u_0\in V(G)$ that is non-adjacent to $v$.
 Add the edge $u_0v_0$.
\item We perform bisplits at $v$ and at a second non-adjacent vertex $u$, obtaining outer vertices $v_1$ and $v_2$ and inner vertex $v_0$ from the first bisplit and outer vertices $u_1$ and $u_2$ and inner vertex $u_0$ from the second.   Add the edge $u_0v_0$.
\item We bisplit $v$, obtaining the inner vertex $u_0$, and outer vertices $u_1$ and $u_2$. We bisplit $u_1$, obtaining an inner vertex $v_0$ and outer vertices $v_1$ and $v_2$, where 
$v_1$ is adjacent to  $u_0$.  Add the edge $u_0v_0$.
\end{enumerate}

The \emph{fundament} of the extension depends on the subtype: 
For 1.\ the fundament is comprised of $u_0,v$ plus any choice
among the vertices of $G$
 of 
two neighbours of $v_1$ and of two neighbours of $v_2$; for 2.\
it will be $u,v$ together with any two neighbours for each of $u_1,u_2,v_1,v_2$
that lie in $G$;
and for 3.\ we choose $v$, one neighbour of $v_1$ and two of each of $u_2$ and $v_2$,
all of them vertices of $G$.

\begin{figure}[ht]
\centering
\includegraphics[scale=1]{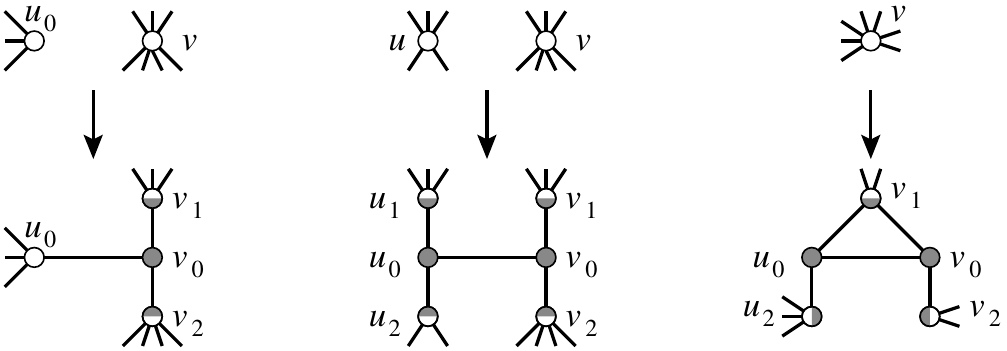}
\caption{Strict linear extension}\label{strlinext}
\end{figure}

Next, assume $u,v,w$ to be three vertices of $G$, so that $w$ is a neighbour
of $u$ but not of $v$. Bisplit $u$, and denote by $u_2$ the new outer 
vertex that is adjacent to $w$, by $u_1$ the other outer vertex 
and by $u_0$ the new inner vertex. 
Subdivide the edge $u_2w$ twice, so that it becomes a  
a path $u_2abw$, where $a$ and $b$ are new vertices. 
Let $G'$ be the graph obtained by adding the edges $bu_0$ and $av$; 
see Figure~\ref{bilinearext}. We say that $G'$ is a \emph{bilinear extension of $G$}. 
Its \emph{fundament} consists of $u,v,w$ together with one neighbour 
of $u_2$, neither $a$ nor $u_0$, and two neighbours of $u_1$, none equal to $u_0$.

\begin{figure}[ht]
\centering
\includegraphics[scale=1]{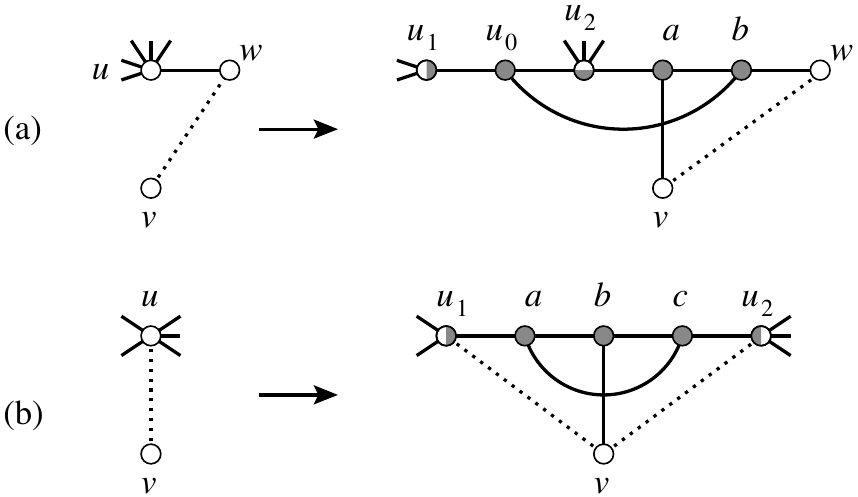}
\caption{(a) Bilinear extension (b) Pseudolinear extension}\label{bilinearext}\label{pseudolinext}
\end{figure}

A graph $G'$ is called a \emph{pseudolinear extension of $G$} if it may be obtained 
from $G$ in the following way. Choose a vertex $u$ of $G$ of degree at least~$4$, 
and a non-neighbour $v$ of $u$. Partition the neighbours of $u$ into 
two sets $N_1$ and $N_2$ each of size at least two. Replace the vertex $u$ 
by two new ones, $u_1$ and $u_2$, so that $u_1$ is adjacent to every vertex in $N_1$
and $u_2$ to every one in $N_2$. Add three new vertices $a,b,c$ and a path $u_1abcu_2$,
and let the graph resulting from adding the edges $ac$ and 
$bv$ be $G'$; see Figure~\ref{pseudolinext}.
We define the \emph{fundament} as $u,v$ plus two neighbours of each of $u_1$ and $u_2$, 
all chosen among $V(G)$.

The penultimate extension is the \emph{quasiquadratic extension}, shown in 
Figure~\ref{quadraticext}. Let $u$ and $v$
be two distinct vertices of $G$, and let $x$ and $y$ be not necessarily distinct vertices
so that $x\neq u$, $y\neq v$ and $\{u,v\}\neq \{x,y\}$. If $u$ and $v$ are adjacent, delete
the edge between them. Add two adjacent new vertices $u'$ and $v'$ and join $u'$ by an edge to $u$ and $x$, 
and make $v'$ adjacent to $v$ and $y$. 
The resulting graph $G'$ is a quasiquadratic extension of $G$.

Norine and Thomas distinguish those quasiquadratic extensions in which 
the edge $uv$ was present in $G$, calling these extensions
 \emph{quadratic}. 
As we will mostly be concerned
with non-quadratic quasiquadratic extensions, let us call these extensions \emph{conservative-quadratic}.
 Thus, in a conservative-quadratic extension the vertices
$u$ and $v$ are not adjacent in $G$, and, in particular, $G$ is an induced subgraph of $G'$.
Let us remark rightaway that, as a conservative-quadratic extension is not quadratic its name is ill-chosen. 
To be more correct, we should call such an extension conservative-quasiquadratic. But life is
far too short for such a long name. 

The \emph{fundament} of the quasiquadratic extension is simply $\{u,v,x,y\}$. 
For later use, let us call $\{u,v\}$ 
the  \emph{upper fundament} of the extension.

\begin{figure}[ht]
\centering
\includegraphics[scale=1]{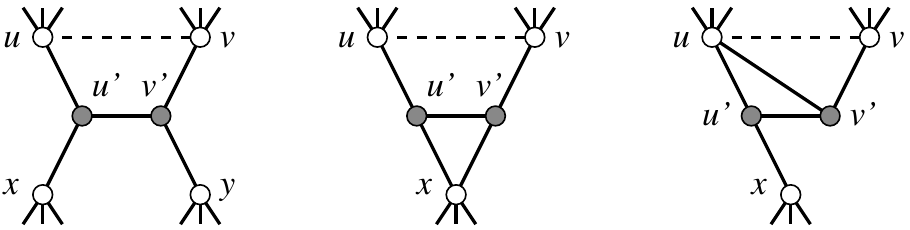}
\caption{(Quasi-)quadratic extension with different allowed identifications}\label{quadraticext}
\end{figure}

Finally, consider distinct vertices $u,v$ and distinct vertices $x,y$ so that $u\neq y$, $v\neq x$ and 
$\{u,v\}\neq\{x,y\}$. If present, delete the edges $uv$ and $xy$. We add four new vertices $u',v',x',y'$
and edges between them so that $u'v'y'x'u'$ is  a $4$-cycle. The graph obtained by adding the edges 
$uu'$, $vv'$, $xx'$ and $yy'$ is a \emph{quasiquartic extension of $G$}. 
%It is a \emph{quartic extension} if $uv$ and $xy$ are edges in $G$.
Its \emph{fundament} consists of $u,v,x,y$.

\comment{In Minimal Bricks sind die Identifizierungen $u=y$ und $v=x$ erlaubt; das muss ein Fehler sein.
In Generating Bricks sind sie auch explizit ausgeschlossen.}

\comment{ARgh! In Norine \& Thomas, Minimal Bricks, hei\ss t eine quasiquadratic ext quadratic, wenn $u$ und $v$
{\bf nicht} benachbart -- in Norine \& Thomas, Generating bricks, wird f\"ur eine quadratic ext vorausgesetzt, 
dass $u$ und $v$ benachbart sind. F\"ur quartic schreiben sie in Minimal Bricks, dass $uv$ und $xy$ Kanten von 
$G$ sein sollen, so steht's auch in Generating bricks. Ich bezieh mich hier auf Preprint-Versionen -- checken, wie's
im publizierten Artikel steht!}

\begin{figure}[ht]
\centering
\includegraphics[scale=1]{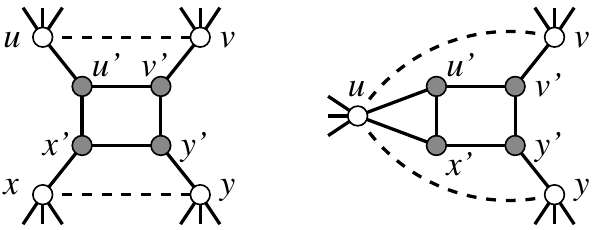}
\caption{(Quasi-)quartic extension with different allowed identifications}\label{quarticext}
\end{figure}

Now, an extension is called \emph{strict} if it is any of the following:  quasiquadratic, quasiquartic, bilinear, pseudolinear, and strict linear.
\comment{Bei einer strict linear ext is einfaches Hinzufügen einer Kante nicht erlaubt: eine linear aber nicht strict linear
ext ist genau das Hinzufügen einer Kante.}
We write $G\to G'$ if $G$ is a brick and $G'$ is obtained from $G$ by a strict extension.

%If the strict extension is quasiquadratic or quasiquartic then $V(G)$ is a subset of $V(G')$. In the other extensions this is not so as vertices that were bisplit vanish. 
%For practical reasons we will force $V(G)\subseteq V(G')$ by  assigning one 
%of the new outer vertices of a bisplit to inherit the identity of the bisplit vertex. 
%More precisely, in strict linear, bilinear and pseudolinear extensions 
% we will identify, if present, $v_1$ with $v$ and $u_1$ with $u$.
%Then we say that $V(G')\sm V(G)$ are the \emph{new vertices} of the strict
%extension   $G\to G'$.

\medskip

Let $F$ be the fundament of the strict extension $G\to G'$.
We observe two trivial properties: 
\begin{equation}\label{funddeg}
\begin{minipage}[c]{0.8\textwidth}\em
Any vertex outside $F$
has the same degree in $G$ as in $G'$.
\end{minipage}\ignorespacesafterend 
\end{equation} 
\begin{equation}\label{fundsize}
\begin{minipage}[c]{0.8\textwidth}\em
We have ${|F|}\leq 3\cdot (|V(G')|-|V(G)|)$. 
\end{minipage}\ignorespacesafterend 
\end{equation} 
\comment{
\begin{tabular}{lccccccc}
            & lin 1 & lin 2 & lin 3 & bilin & pseudo & quad & quart \\ 
\# fund     &  6    &   10  &   6   &  6    &    6   &  4   &   4   \\ 
\# new vxs  &  2    &   4   &   4   &  4    &    4   &  2   &   4   \\
ratio       & $3$   & $5/2$ & $3/2$ & $3/2$ &  $3/2$ & $2$  &  $1$
\end{tabular}
}
We note that the ratio $3$ is attained for strict linear
extensions of the first type: There the fundament consists of $u,v$ 
plus four neighbours of $v$, while $G'$ has only two vertices more than $G$.

It is easy to see that a strict extension $G'$ of a brick $G$ is $3$-connected. Also, it is not difficult to find a perfect matching of $G'-x-y$ for any pair of vertices $x,y\in V(G')$, with exception of the pair $u_0,v_0$ if $G\to G'$ is a strict linear extension, and the pair $u_0b$, or $ac$, if $G\to G'$ is a bi- or pseudolinear extension, respectively. These particular cases can be reduced to the exercise of finding a perfect matching in the graph obtained from $G$ by bisplitting a vertex, deleting the new inner vertex and another vertex distinct from the new outer vertices. 
%second picture of strict linear extension can also be reduced to this, assuming we already proved it for the first picture
 Using Tutte's theorem, and the fact that $G$ is brick, this is not hard to solve.

%The following observation is claimed to be true and ``not hard to see'' in 
This leads to the following lemma, which has also been observed by 
Norine and Thomas~\cite{NoTh06}: 
\begin{lemma}\label{bricktobrick}
Any strict extension of a brick is a brick.
\end{lemma}

We close this section with an example. In Figure~\ref{tripleladderbrick}
we build up a triple ladder by repeatedly alternating between 
quasiquartic and quasiquadratic extensions, starting from a prism. As by Lemma~\ref{bricktobrick}, 
strict extensions take a brick to a brick, we deduce that the triple ladder
is a brick. To see that it is 
a minimal brick, note that the deletion of any edge 
 results in a graph that fails to be $3$-connected.
\begin{figure}[ht]
\centering
\includegraphics[scale=0.7]{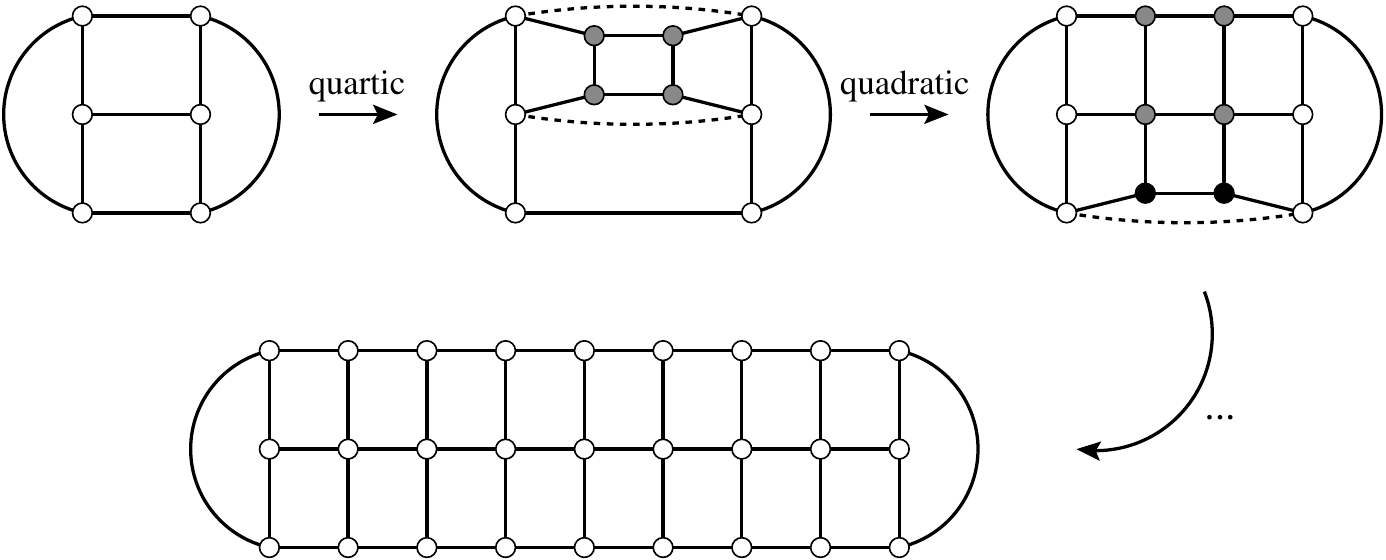}
\caption{A minimal brick}\label{tripleladderbrick}
\end{figure}

\section{Brick on brick}

We will call a sequence 
$G_0{\to}G_1{\to}\ldots{\to}G_k$
a \emph{brick-on-brick sequence} if all the $G_0,\ldots,G_k$ are bricks
(not necessarily minimal) and if all the $G_{i-1}{\to}G_i$ are strict
extensions. Thus, the theorem of Norine and Thomas states
that every minimal brick $G$ has such a brick-on-brick sequence
that starts with $K_4$ or the prism and ends with $G$, 
and in which all intermediate bricks are minimal---unless $G$ is the 
Petersen graph.

We formulate a simple lemma that allows us to
reorder a brick-on-brick sequence.
\begin{lemma}\label{reorderlem}
Let $A\to B\to C$ be a 
brick-on-brick sequence, so that $A\to B$ is con\-ser\-va\-tive-quadratic
with new vertices $p,q$ and
so that $p,q$ do not lie in the fundament of $B\to C$.
Then there exist a brick $B'$  
 so that 
$A\to B'\to C$
is a brick-on-brick sequence and $B'\to C$ is conservative-quadratic
with new vertices $p,q$.
\end{lemma}

\begin{proof}
Since $A\to B$ is con\-ser\-va\-tive-quadratic, we have that $B-\{p,q\}=A$. It is easy to verify that thus $A\to C-\{p,q\}$ is a strict extension (of the same type as $B\to C$). 
For this, it is important to note that by assumption,  $p$ and $q$ are not in the fundament of $B\to C$. In particular, any bisplittings of $B\to C$ can also be performed in $A$
at vertices of degree~$\geq 4$. Using Lemma~\ref{bricktobrick}, we see that $B':=C-\{p,q\}$ is a brick.

It remains to show that $B'\to C$ is a conservative-quadratic extension. This is easy to check if
 none of the vertices of the fundament $F$ of $A\to B$ has suffered a bisplit during the operation
$A\to B'$. 
 So assume
there is a vertex $s\in F$ which is bisplit in $A\to B'$, and say $s$ is adjacent to $p$ in $B$. Then, however, $s$ is also bisplit
in $B\to C$, and in $C$, one of the new outer vertices, say $s_1$, is adjacent to~$p$. 
So  $B'\to C$ is a quasiquadratic extension. Note that the number of edges gained in $A\to B'$ and in $B\to C$ is the same, and so, with $A\to B$, also $B'\to C$ is 
 conservative-quadratic.
%alternativ: Indeed, in all strict extensions if a vertex gains a new adjacency then the new neighbour is a new vertex. 
\end{proof}

Let us now examine how the edge density changes in  a 
brick-on-brick sequence. 
Suppose $G=(V,E)$ is a minimal brick other than the Petersen graph,
and let $G_0\to\ldots\to G_k$ be a brick-on-brick sequence for $G$ 
as given by Theorem~\ref{generatingminbricks}, that is, $G=G_k$ and $G_0$
is either the $K_4$ or the prism.
For a set of indices $I\subseteq\{1,\ldots,k\}$ we define 
$\nu(I)$
 to be the total number of vertices added in extensions corresponding 
to $I$: 
\[
\nu(I):=\sum_{i\in I}(|V(G_{i})|-|V(G_{i-1})|).
\]

 Similarly, 
we define 
\[
\epsilon(I):=\sum_{i\in I}(|E(G_{i})|-|E(G_{i-1})|).
\]

Now, let $I_1$ be the set of indices $i\in\{1,\ldots,k\}$ for which 
$G_{i-1}\to G_i$ is a strict linear, bilinear or pseudolinear extension, and  set 
$\nu_1=\nu(I_1)$ and $\epsilon_1=\epsilon(I_1)$. We define analogously
$I_2$, $\nu_2$ and $\epsilon_2$ (resp.\ $I^c_2$, $\nu^c_2$ and $\epsilon^c_2$)
for quasiquadratic (resp.\ conservative-quadratic) extensions and 
$I_3$, $\nu_3$ and $\epsilon_3$
for quasiquartic extensions. 

Finally, let $\nu_0:=|V(G_{0})|$ and $\epsilon_0:=|E(G_{0})|$. As $G_0$ is either 
$K_4$ or the prism it follows that $(\nu_0,\epsilon_0)\in\{(4,6),(6,9)\}$.
Moreover, we clearly have that 
\begin{equation}\label{summe}
|V(G)|=\nu_0+\nu_1+\nu_2+\nu_3\text{ and }|E(G)|=\epsilon_0+\epsilon_1+\epsilon_2+\epsilon_3.
\end{equation}
It is easy to calculate that 
\begin{equation}\label{werte}
\text{
$\epsilon_0=\frac{3}{2}\nu_0$,
 $\epsilon_1\leq \frac{3}{2}\nu_1$, $(\epsilon_2-\epsilon^c_2)=\frac{4}{2}(\nu_2-\nu^c_2)$, $\epsilon^c_2=\frac{5}{2}\nu^c_2$ and
 $\epsilon_3\leq\frac{8}{4}\nu_3$. 
}
\end{equation}

From~\eqref{werte}, we see that the `edge density gain' is largest when performing 
conservative-quadratic extensions. 
In fact, the greater the average degree of a minimal brick, the more 
conservative-quadratic extensions must have been used in any of its brick-on-brick sequences:

\begin{lemma}\label{lotsofquad}
Let $\delta>0$, and let $G$ be a minimal brick 
with average degree $d(G)\geq 4+\delta$.
For any brick-on-brick sequence $G_0\to\ldots\to G_k$ 
with $G=G_k$ and $G_0\in\{K_4,\text{Prism}\}$ it holds that 
$\nu^c_2\geq\delta|V(G)|$.
\end{lemma}
\begin{proof}
Let $G=(V,E)$. Using~\eqref{summe} and~\eqref{werte}, we find that
\begin{align*}
\frac{4+\delta}{2}&\leq \frac{|E|}{|V|}\\
& =\frac{1}{|V|}\left(\epsilon_0+\epsilon_1+\epsilon_2+\epsilon_3\right)
\\ &\leq \frac{1}{|V|}\left(\frac{3}{2}\nu_0+\frac{3}{2}\nu_1+2(\nu_2-\nu^c_2)
+\frac{5}{2}\nu^c_2+2\nu_3\right)\\
&\leq \frac{1}{|V|}\left(2|V|+\frac{1}{2}\nu^c_2\right),
\end{align*}
and consequently, $\nu^c_2\geq \delta |V|$.
\end{proof}

We postpone the proof of the following lemma to the next section. 
\begin{lemma}\label{quadonquad}
Let 
$G$ be a brick, and let $G''$ be  a conservative-quadratic extension 
 of a conservative-quadratic extension $G'$ of $G$.
Let $u'$ and $v'$ be the new vertices of $G'$. If one of $u'$, $v'$ is used for the fundament of $G'\to G''$ then $G''$ is not a minimal brick.
\end{lemma}

\begin{lemma}\label{highavdeg}
Every minimal brick $G$ of average degree $d(G)\geq 4+\delta$ with $\delta>0$
has at least $(4\delta-3)|V(G)|$ vertices of degree~$3$.
\end{lemma}

\begin{proof}
By Theorem~\ref{generatingminbricks}, there is a brick-on-brick sequence 
 $\mathcal B:=G_0\to \ldots\to G_k$ for $G$, 
where all intermediate graphs are minimal bricks. 
With Lemma~\ref{lotsofquad} we find that 
\begin{equation}\label{nu2strich}
\nu^c_2\geq \delta  |V(G)|.
\end{equation}

This means that there is a set $Q$ of at least $\delta |V(G)|$ vertices that arise as new vertices in some conservative-quadratic extension of $\mathcal B$. 
Denote by $Q_1$ the set of those vertices in $Q$ that are used in the fundament of any later extension  of $\mathcal B$, and let $Q_2:=Q\setminus Q_1$. Then $Q_2\subseteq  V(G)$ and the vertices of  $Q_2$ have degree~$3$ in $G$ by~\eqref{funddeg}. 

Hence if $|Q_2|\geq (4\delta -3)|V(G)|$, then we are done. So assume otherwise. Then
\begin{equation}\label{vielefundis}
|Q_1| = |Q| -|Q_2|> \delta |V(G)|-(4\delta-3)|V(G)|=3(1-\delta)|V(G)|.
\end{equation}

Let $I$ be the set of indices of extensions of $\mathcal B$ that use some vertex of $Q_1$ in their fundament which has not been used in the fundament of earlier extensions of $\mathcal B$. Then~\eqref{fundsize} together with~\eqref{vielefundis} implies that $\nu (I) > (1-\delta)|V(G)|$.

This means that by~\eqref{summe} and by~\eqref{nu2strich}, there is an index $j\in I$ that corresponds to a con\-ser\-vative-quadratic extension  $G_{j-1}\to G_j$ of $\mathcal B$. 
Let  $q\in Q_1$ lie in the fundament of this extension. 

We apply Lemma~\ref{reorderlem} repeatedly in order to finally obtain a brick $G'_{j-2}$
so that 
\[
G'_{j-2}\to G_{j-1}\to G_j
\]
is a brick-on-brick sequence, with $q$ being one of the new vertices in 
the con\-ser\-vative-quadratic extension $G'_{j-2}\to G_{j-1}$. 
This contradicts Lemma~\ref{quadonquad}.
\end{proof}

We are now ready to prove our main theorem.

\begin{proof}[Proof of Theorem~\ref{thmdeg4}]
%Set $\alpha:=\frac{1}{9}$.
Given a minimal brick $G$ we distinguish two cases. If the average degree 
of $G$ is at least $4+\frac{7}{9}$, then we apply Lemma~\ref{highavdeg}
to see that at $\frac{1}{9}|V(G)|$ of the vertices have degree~$3$. 

So, we may assume that $G$ has average degree  at most $5- \frac{2}{9}$. 
Denote by $V_{\leq 4}$ the set of all vertices of degree at most~$4$, and by 
$V_{\geq 5}$ the set of all vertices of degree at least~$5$. 
Then
\[
\left(5-\frac{2}{9}\right)|V(G)|\geq \sum_{v\in V(G)}d(v)\geq 3|V_{\leq 4}|+5|V_{\geq 5}|
=5|V(G)|-2|V_{\leq 4}|,
\]
which leads to $|V_{\leq 4}|\geq \frac{1}{9}|V(G)|$.
 
In either case we find that at least a ninth of the vertices of $G$ have degree at most~$4$.
\end{proof}

\section{Proof of Lemma~\ref{quadonquad}}

\newcommand{\case}[1]{\smallskip\noindent\textbf{#1:}}

We dedicate this section entirely to the proof of Lemma~\ref{quadonquad}.
We shall use the notation from Figure~\ref{twicequad}, that is, $\{x,u,v,y\}$ is the 
fundament of the conservative-quadratic extension $G\to G'$, and 
$\{u',r,s,t\}$ is the fundament of the conservative-quadratic extension $G'\to G''$, with new vertices $r'$ and $s'$, where $r'$ is adjacent to $u'$ and $r$, and $s'$ is adjacent to $s$ and~$t$. Several of these vertices may be identified, some of them are by
definition distinct:
\begin{equation*}%\label{noid}
u\neq x,\, v\neq y,\, u'\neq r,\, s\neq t,
\emtext{ and } u',v',r',s' \emtext{ are pairwise distinct. }
\end{equation*}

\begin{figure}[ht]
\centering
\includegraphics[scale=1]{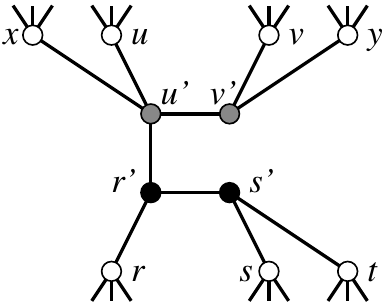}
\caption{Applying two quadratic extension on top of each other.
Note that several of the vertices in the figure might be identified. 
%(In particular, the top line might contain only $3$ vertices, and the vertices from the lowest line may be identical to vertices from the two top lines.)
}\label{twicequad}
\end{figure}

Assume for contradiction that $G''$ is a minimal brick. 
From this we will deduce that $G''-uu'$ is bicritical and $3$-connected, 
which is clearly impossible.

\medskip

We start by  proving that
\begin{equation}\label{vstrich}
\text{ $v'\notin \{s,t\}$}.
\end{equation}
Indeed, suppose otherwise. Then, as $G''$ is a conservative-quadratic  
extension of $G'$, the graph
$G''-v'u'$ is a quadratic extension of $G'$ (with $v'$ and $u'$ constituting the upper fundament of $G'\to G''-v'u'$). Hence $G''-v'u'$ is a brick. Thus $G''$ is not minimal, against our assumption.

Next, we show that
\begin{equation}\label{oneoftwo}
\emtext{
$G''-uu'$ is bicritical.
}
\end{equation}

For this, let $a,b$ be two vertices of $G''$. 
Our aim is to find a perfect matching of $G''-a-b$ that avoids $uu'$.
We may assume that neither of $a$, $b$
is incident with $uu'$ as otherwise any perfect matching of $G''-a-b$ serves for our purpose (and exists since $G''$ is bicritical). Therefore, the following three cases cover all possible cases (after possibly swapping $a$ and $b$).

\case{Case 1} $a,b\in V(G)$.

 Since $G$ is bicritical there is a perfect
matching $M$ of $G-a-b$. This matching together with the edges $u'v'$ and
$r's'$ yields a perfect matching of $G''-a-b$ that avoids $uu'$.

\case{Case 2} $a\in V(G)$ but $b\notin V(G)$.

Our aim is to find a substitute $b'\in V(G)$ different from $a$, so
that a perfect matching $M$ of $G-a-b'$ together with two edges $e,f$
form a perfect matching $M'$ of $G''-a-b$ that avoids $uu'$. 

{\em Subcase:}
$b=s'$. As $v\neq y$, we may choose
  $b'\in\{v,y\}$ distinct from $a$,
and let $M':=M+u'r'+v'b'$.

{\em Subcase:}
 $b\in\{v',r'\}$ and  $\{s,t\}\neq\{u',a\}$. 
 Choose
$b'\in\{s,t\}$  distinct from $u'$ and $a$, and note that~\eqref{vstrich} implies that $b'\neq v'$.
Then the matching $M':=M+u'v'+s'b'$ if $b=r'$
or  the matching  $M':=M+u'r'+s'b'$ if $b=v'$
 is as desired.

{\em Subcase:}
 $b=v'$ and  $\{s,t\}=\{u',a\}$. 
 Note that then $r\neq a$. Choose
$b':=r$.
Then the matching $M':=M+u's'+r'b'$  is as desired.

{\em Subcase:}
 $b=r'$ and  $\{s,t\}=\{u',a\}$. 
 Choose
$b'\in\{v,y\}$  distinct from $a$.
Then the matching $M':=M+u's'+v'b'$  is as desired.

\case{Case 3} $a,b\notin V(G)$. 

Then $a,b\in \{v',r',s'\}$. 
If $\{a,b\}=\{r',s'\}$ we take a perfect matching $M$ of $G$ plus $u'v'$. 
If $\{a,b\}=\{v',s'\}$ we  choose a perfect matching of $G$ 
together with $u'r'$. 

It remains the case when $\{a,b\}=\{r',v'\}$.
We choose  $b'\in\{s,t\}$  distinct from $x$. By~\eqref{vstrich} we have $b'\neq v'$. If $b'=u'$ then a perfect matching
of $G$ together with $u's'$ is as desired. Otherwise, we can use
a perfect matching
of $G-x-b'$ together with $u'x$ and $s'b'$.

We have thus proved~\eqref{oneoftwo}.

\medskip

We finish the proof of the lemma by showing that
\begin{equation}\label{3zsh}
\emtext{
 $G''-uu'$ is $3$-connected.
}
\end{equation}

%Suppose otherwise. Then there are vertices $w,z$ such that $G''-uu'$ has a $2$-separation $(A,B)$ with $\{w,z\}= A\cap B$. Since $G''$ is $3$-connected we know that $w\neq z$ and that $uu'$ goes from $A\setminus B$ to $B\setminus A$. Say $u\in A\sm B$ and $u'\in B\sm A$.

Suppose otherwise. 
Then there are vertices $w,z$ such that $G''-uu'-w-z$ is disconnected. In other words, $G''-uu'$ is the union of two subgraphs $A,B$ with $\{w,z\}= V(A\cap B)$ and $A\sm B\neq \emptyset\neq B\sm A$. Since $G''$ is $3$-connected we know that $w\neq z$ and that $uu'$ goes from $A\setminus B$ to $B\setminus A$. Say $u\in A\sm B$ and $u'\in B\sm A$.
Since $G''$ is $3$-connected we know that $w\neq z$ and that $uu'$ goes from $A- B$ to $B- A$. Say $u\in V(A-B)$ and $u'\in V(B-A)$.

Since $G$ is $3$-connected,  all of $G$ is contained  
in either $A$ or in $B$.
As $u\notin V(B)$, it must be that $G\subseteq A$. 
Thus $x$, as a neighbour of $u'$, lies in $A\cap B=\{w,z\}$. Say $x=w$.

Now, either $v'=z$ or $v'\in B\sm A$ and $\{v,y\}=\{x, z\}$ (then, in particular, one of $v,y$ is equal to $z$). In either case, 
we find that the the new vertices $r',s'$ of $G''$ lie in $B-A$. 
For the neighbours $r,s,t$ of $r'$ and $s'$ we deduce that
\begin{equation}\label{rst}
\{r,s,t\}\subseteq \{u',v',x, z\}. 
\end{equation}

We claim that $x\notin\{s,t\}$. Otherwise, say if $x=s$, we perform a quadratic extension 
in $G'$ with fundament $\{u',x,r,t\}$ and upper fundament $\{u',x\}$. In 
this way, the edge $u'x$ vanishes in the 
quadratic extension $G'\to G''-xu'$. Thus, $G''-xu'$ is a brick, which contradicts
the minimality of $G''$.
With a similar reasoning we see $v'\notin\{s,t\}$.

Thus, one of $s,t$ must be equal to $u'$, and the other equal to $z$ (recall that $u'\neq z$).
As the fundament of the quasiquadratic extension $G'\to G''$, the set
$\{r,s,t,u'\}$ contains at least three vertices. 
 By~\eqref{rst},  this implies that $r$ is either $x$ or $v'$. 
But then, either choosing $u'$ and $x$ as the upper fundament of the quadratic extension $G'\to G''-xu'$, or choosing $u'$ and $v'$ as the upper fundament of the quadratic extension $G'\to G''-u'v'$, we find a contradiction to the fact that $G''$ is a minimal brick. 
This concludes the proof of~\eqref{3zsh}.

\section{Discussion}

In this work, we proved that in a minimal brick the number of vertices of degree~$\leq 4$ is a positive fraction of the total 
number of vertices.
 On the other hand, if we look for large degree vertices in a minimal brick, it is not difficult to find examples with
a  few vertices of arbitrary large degree (for instance even wheels). 
It seems less evident that one can also construct minimal bricks with many vertices of degree $\geq 5$. 
We provide an example in Figure~\ref{ladderplus}, where about a seventh of the vertices  have degree~$6$. 
This graph is a indeed brick, since it can be built from the triple ladder of Figure~\ref{tripleladderbrick}
by performing two quadratic extensions at triples like $r,s,t$. It is a minimal brick as clearly
every edge is necessary for $3$-connectivity. 
\begin{figure}[ht]
\centering
\includegraphics[scale=0.7]{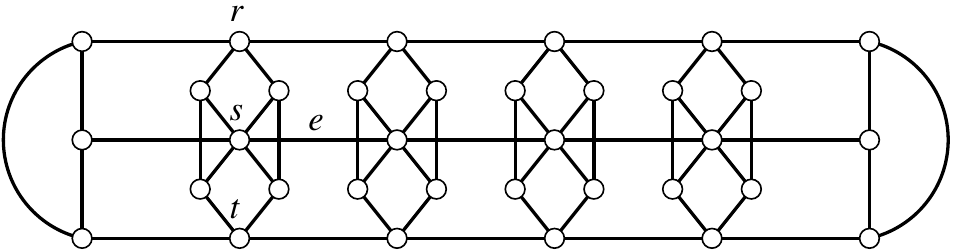}
\caption{A minimal brick}\label{ladderplus}
\end{figure}

Vertices of degree $\leq 4$ and even cubic vertices  seem to be abundant in all examples. 
In the example with fewest proportion of degree~$3$ vertices we know, the triple ladder
in Figure~\ref{tripleladderbrick}, they still make up two thirds of the vertices. In
that respect, our result with a fraction of $\geq \frac{1}{9}$ of the vertices seems quite low.

\smallskip

The main aim of this paper was to develop ideas and techniques that ultimately should serve
to settle the Norine-Thomas conjecture. While we believe to have done a substantial step 
in that direction, there are still serious obstacles lying on that route. Let us briefly
outline some of them.

Clearly, an average degree of at most $4-\gamma$ (for some small constant $\gamma>0$)
yields a positive fraction of degree~$3$ vertices. We may therefore assume that our 
minimal bricks have average degree of about $4$ and higher. 
While an average degree of about $5$ and higher leads to a brick-on-brick sequence with many  conservative-quadratic extensions (cf.~Lemma~\ref{lotsofquad}),  the now lower bound on the average degree will give us less information on the kind of extensions our brick-on-brick sequence is composed of. In particular,  quadratic  and conservative-quartic (those that do not involve edge deletions) might appear, as they push the average degree towards $4$. Even worse, because
conservative-quadratic extensions yield a relatively large edge-density increase, we may also 
have lots of strict linear, bilinear or pseudolinear extensions.

To handle this, we would seem to need a much stronger version of Lemma~\ref{quadonquad},
that also forbids two chained quadratic extensions, say, that increase the degree
of a fundament vertex. Unfortunately, two such extension might actually occur while still yielding a minimal
brick: This is exactly what happened to produce the degree~$6$ vertices in
Figure~\ref{ladderplus}.

\section{Acknowledgment}
The second author would like to thank Andrea Jim\' enez for inspiring discussions, and for drawing her attention to the subject.

\bibliographystyle{amsplain}
\bibliography{brickbib}

\providecommand{\bysame}{\leavevmode\hbox to3em{\hrulefill}\thinspace}
\providecommand{\MR}{\relax\ifhmode\unskip\space\fi MR }
\providecommand{\MRhref}[2]{%
  \href{http://www.ams.org/mathscinet-getitem?mr=#1}{#2}
}
\providecommand{\href}[2]{#2}
\begin{thebibliography}{1}

\bibitem{CLM06}
M.H. de~Carvalho, C.L. Lucchesi, and U.S.R. Murty, \emph{How to build a brick},
  Disc.\ Math. \textbf{306} (2006), 2386--2410.

\bibitem{ELP}
J.~Edmonds, L.~Lov{\'a}sz, and W.R.Pulleyblank, \emph{Brick decomposition and
  the matching rank of graphs}, Combinatorica \textbf{2} (1982), 247--274.

\bibitem{BricBracDecomp}
L.~Lov{\'a}sz, \emph{Matching structure and the matching lattice}, J.~Combin.\
  Theory (Series B) \textbf{43} (1987), 187--222.

\bibitem{LP86}
L.~Lov{\'a}sz and M.D. Plummer, \emph{Matching theory}, Akad{\'e}miai Kiad{\'o}
  - North Holland, 1986.

\bibitem{NoTh06}
S.~Norine and R.~Thomas, \emph{Minimal bricks}, J.~Combin.\ Theory (Series B)
  \textbf{96} (2006), 505--513.

\bibitem{NoTh07}
\bysame, \emph{Generating bricks}, J.~Combin.\ Theory (Series B) \textbf{97}
  (2007), 769--817.

\bibitem{LexBible}
A.~Schrijver, \emph{Combinatorial optimization. {P}olyhedra and efficiency},
  Springer-Verlag, 2003.

\end{thebibliography}

\small
\vskip2mm plus 1fill
Version 2 Oct 2012
\bigbreak

\noindent
\begin{tabular}{cc}
\begin{minipage}[t]{0.5\linewidth}
Henning Bruhn
{\tt <bruhn@math.jussieu.fr>}\\
Combinatoire et Optimisation\\
Universit\'e Pierre et Marie Curie\\
4 place Jussieu\\
75252 Paris cedex 05\\
France
\end{minipage}
&
\begin{minipage}[t]{0.5\linewidth}
Maya Stein
{\tt <mstein@dim.uchile.cl>}\\
Centro de Modelamiento Matem\'atico\\
Universidad de Chile\\ 
Blanco Encalada, 2120\\ 
Santiago\\
Chile
\end{minipage}
\end{tabular}

\end{document}